\documentclass[11pt]{article}

\usepackage{amsfonts}
\usepackage{amsmath}
\usepackage{epic,eepic}
\usepackage{stmaryrd}
\usepackage{ifthen}
\usepackage{latexsym,amssymb}

\newcommand{\df}{\stackrel{\rm def}{=}}
\newcommand{\Hom}{{\rm Hom}}
\newcommand{\rn}{\boldsymbol}
\newcommand{\scr}{\mathcal}
\newcommand{\eval}[2]{\llbracket #1 \rrbracket_{#2}}

\newtheorem{theorem}{Theorem}[section]
\newtheorem{lemma}[theorem]{Lemma}

\newtheorem{claim}[theorem]{Claim}

\newtheorem{conjecture}[theorem]{Conjecture}

\newcommand{\beginproof}{\medskip\noindent{\bf Proof.~}}
\newcommand{\beginproofof}[1]{\medskip\noindent{\bf Proof of #1.~}}
\newcommand{\beginproofdotless}{\medskip\noindent{\bf Proof}}
\newcommand{\finishproof}{\hspace{0.2ex}\rule{1ex}{1ex}}
\newenvironment{proof}{\beginproof}{\unskip\nolinebreak\finishproof\par\medskip}

\newcommand{\hh}{h}
\newcommand{\hhh}[1]{h^{(#1)}}

\begin{document}


\newcommand{\case}[1]{\medskip\noindent{\bf Case #1} }
\newcommand{\beq}[1]{\begin{equation}\label{#1}}
\newcommand{\eeq}{\end{equation}}
\newcommand{\req}[1]{(\ref{#1})}

\newcommand{\ex}{\mathrm{ex}}
\newcommand{\C}[1]{{\scr #1}}
\newcommand{\B}[1]{{\bf #1}}
\newcommand{\I}[1]{{\mathbb #1}}
\renewcommand{\O}[1]{\overline{#1}}
\newcommand{\ceil}[1]{\lceil #1\rceil}
\newcommand{\e}{\varepsilon}
\newcommand{\floor}[1]{\lfloor #1\rfloor}
\newcommand{\me}{{\mathrm e}}
\renewcommand{\mid}{:}
\newcommand{\rcpc}[1][nothing]{\ifthenelse{\equal{#1}{nothing}}
 {\cite{razborov:08}}{\cite[#1]{razborov:08}}}
\newcommand{\rjsl}[1][nothing]{\ifthenelse{\equal{#1}{nothing}}
 {\cite{razborov:07}}{\cite[#1]{razborov:07}}}
\newif\ifnotesw\noteswtrue
\newcommand{\comment}[1]{\ifnotesw $\blacktriangleright$\ {\sf #1}\
  $\blacktriangleleft$ \fi}
\newcommand{\hide}[1]{}

\bibliographystyle{amsalpha}

\title{Asymptotic Structure of Graphs with\\ the Minimum Number of Triangles}

\author{Oleg Pikhurko\footnote{Supported by ERC
grant~306493 and EPSRC grant~EP/K012045/1.}\\
Mathematics Institute and DIMAP\\
University of Warwick\\
Coventry CV4 7AL, UK
\and
Alexander Razborov\footnote{Part of this work was done while
the author was at Steklov Mathematical Institute, supported by the Russian Foundation
for Basic Research, and at Toyota Technological Institute, Chicago.}\\
Department of Computer Science\\
University of Chicago\\
Chicago, IL 60637
}

\maketitle

\begin{abstract}
 We consider the problem of minimizing the  number of triangles in
a graph of given order and size and
describe the asymptotic structure of extremal graphs.
This is achieved by characterizing the set of flag algebra homomorphisms
that minimize the triangle density.
 \end{abstract}

\section{Introduction}

The famous theorem of Tur\'an~\cite{turan:41} determines $\ex(n,K_r)$,
the maximum number of edges in a graph with $n$ vertices that does not
contain the $r$-clique $K_r$ (the case $r=3$ was previously solved by
 Mantel~\cite{mantel:07}). The
unique extremal graph is the \emph{Tur\'an graph} $T_{r-1}(n)$, the
complete $(r-1)$-partite graph of order $n$ whose part sizes differ at
most by 1. Thus, for fixed $r$, we have $\ex(n,K_r)=(1-\frac1{r-1}+o(1))
{n\choose 2}$.

Rademacher (unpublished, 1941) proved that a graph with $\ex(n,K_3)+1$
edges has at least $\floor{n/2}$ triangles. This prompted Erd\H
os~\cite{erdos:55} to pose the more general problem:
what is $g_r(m,n)$, the smallest number of $K_r$-subgraphs in a graph with
$n$ vertices and $m$ edges?
Various results have been obtained by
Erd\H os~\cite{erdos:62,erdos:69:cpm},
Moon and Moser~\cite{moon+moser:62},
Nordhaus and Stewart~\cite{nordhaus+stewart:63},
Bollob\'as~\cite{bollobas:76}, Fisher~\cite{fisher:89},
Lov\'asz and Simonovits~\cite{lovasz+simonovits:76,lovasz+simonovits:83},
Razborov \cite{razborov:07,razborov:08}, Nikiforov~\cite{nikiforov:11},
Reiher~\cite{reiher:Kr}, and others.

Let us consider the asymptotic question, that is,
what is the limit
 $$
 g_r(a)\df\lim_{n\to\infty} \frac{g_r\!\left(\floor{a{n\choose 2}},n\right)}{{n\choose r}}
 $$
 for any given $a\in [0,1]$ and $r$? While it is not difficult to show that the limit
exists, determining $g_r(a)$ is a much harder task that was accomplished
only recently
(for $r=3$ by Razborov~\cite{razborov:08}, for $r=4$ by
Nikiforov~\cite{nikiforov:11}, and for $r\ge 5$ by
Reiher~\cite{reiher:Kr}).

The following construction gives the value of $g_3(a)$ (as well as $g_r(a)$
for every $r\ge 4$).
Given $a\in (0,1)$, we choose integer $t\ge 1$ and real $c\in \left[ \frac 1{t+1}, \frac 1t\right)$ such that the complete $(t+1)$-partite graph
of order $n\to\infty$ with $t$ largest parts each of size $(c+o(1))n$ has edge density $a+o(1)$. Formally,
let integer $t\ge 1$ satisfy
 \beq{t}
 a\in \left(1-\frac1t,1-\frac1{t+1}\right]
 \eeq
 and let real
 \beq{CExplicit}
 c=\frac{t+\sqrt{t(t-a(t+1))}}{t(t+1)}
 \eeq
 be the (unique) root of the quadratic equation
 \beq{c}
 2\left({t\choose 2}c^2+tc(1-tc)\right)=a
 \eeq
 with $c\ge \frac1{t+1}$.
 \hide{Note that
the left-hand side of~\req{c} is strictly decreasing when
$c\ge \frac1{t+1}$ and assumes value $1-\frac1{t+1}\ge a$ when
$c=\frac1{t+1}$ and value $1-\frac1t\le a$ when $c=\frac1t$. (So we have
in fact $c\le
\frac1t$.)}
 Since $a> 1-\frac1t$, it follows from~\req{CExplicit} (or from \req{c}) that $c<
\frac1t$.
Partition the vertex set $[n]=\{1,\dots,n\}$ into $t+1$ non-empty parts
$V_1,\dots,V_{t+1}$ with
$|V_1|=\dots=|V_t|=\floor{cn}$ for $i\in [t]$. Let $G$ be obtained from the
complete $t$-partite graph $K(V_1,\dots,V_{t-1},U)$, where $U=V_{t}\cup V_{t+1}$,
by adding an arbitrary triangle-free graph $G[U]$ on $U$ with $|V_t|\,|V_{t+1}|$
edges\footnote{One possible choice is to take $G[U]=K(V_t,V_{t+1})$, resulting in
$G=K(V_1,\dots,V_{t+1})$. But since each edge of $G[U]$ belongs to exactly
$|V_1|+\dots+|V_{t-1}|$ triangles, the choice of $G[U]$ has no
effect on the triangle density.}. Clearly, the edge density of $G$ is
$a+o(1)$. Thus $g_3(a)\le \hh(a)$, where
 \beq{gr}
 \hh(a)\df 6\left({t\choose 3}c^3+{t\choose 2}c^{2}(1-tc)\right).
 \eeq
 If $a=1$, we let $G$ be the complete graph $K_n$ and define $\hh(1)=1$.
If $a=0$, we take the empty graph and let $h(0)=0$.
For $a\in [0,1]$, let $\C H_{a,n}$ be the set of all possible graphs $G$
on $[n]$ that arise in this way,
$\C H_a\df\cup_{n\in\I N}\C H_{a,n}$, and $\C H\df\cup_{a\in[0,1]} \C H_a$.
In general, $\C H_{a,n}$ has many non-isomorphic graphs and this
seems to be one of the reasons why this
extremal problem is so difficult.

Although each of the papers~\cite{razborov:08,nikiforov:11,reiher:Kr} implies
the lower bound $g_3(a)\ge \hh(a)$, it is not clear how to extract the structural
information about extremal graphs from these proofs.
Here we partially fill this gap by showing that, modulo
changing a negligible proportion of adjacencies, the set $\C H$ consists of all
almost extremal graphs for the $g_3$-problem. Here is the formal statement.

\begin{theorem}\label{th:comb} For every $\e>0$ there are $\delta>0$ and $n_0$ such
that every graph $G$ with $n\ge n_0$ vertices and at most $(g_3(a)+\delta){n\choose 3}$ triangles, where $a=e(G)/{n\choose 2}$,
can be made isomorphic to some graph in $\C H_{a,n}$
by changing at most $\e {n\choose 2}$
adjacencies.\end{theorem}

This theorem is obtained by building upon the flag algebra approach
from~\cite{razborov:08}. In order to prove it we have to characterize
first the set of extremal flag algebra homomorphisms for the
$g_3$-problem. This is done in
Theorem~\ref{th:Phi} of Section~\ref{flag}, where the precise statement can be found.
This task requires some extra work in
addition to the arguments in~\cite{razborov:08} and is an example of how
flag algebra calculations may lead to structural results about graphs.
(For some other results of a similar type, see
e.g.~\cite{pikhurko:11,CKPSTY,das+huang+ma+naves+sudakov:13,hatami+hladky+kral+norine+razborov:13,pikhurko+vaughan:13}.)

Theorem~\ref{th:comb} (or more precisely Theorem~\ref{th:Phi}) can be viewed as a small step towards the more
general problem of understanding graph limits with given edge and triangle densities. The latter problem naturally appears in the study of  exponential random graphs (see e.g.\ \cite{radin+yin:11:arxiv,aristoff+radin:13,chatterjee+diaconis:13,radin+sadun:13,radin+ren+sadun:13})
and large deviation inequalities for the triangle density
in Erd\H{o}s-R\'enyi random graphs (see e.g.\ \cite{chatterjee+dey:10,chatterjee+varadhan:11,chatterjee+dembo:14:arxiv,lubetzky+zhao:rsa,lubetzky+zhao:14:arxiv}).

Our initial motivation was  the following  conjecture of Lov\'asz and Simonovits~\cite[Conjecture~1]{lovasz+simonovits:76} for $r=3$.

\begin{conjecture}\label{cj:LS} For every
$r\ge 3$ there is $n_0$ such that for every $n\ge n_0$ and
$m$ with $0\le m\le {n\choose 2}$
at least one of $g_r(m,n)$-extremal graphs is obtained
from a complete partite graph by adding a triangle-free graph inside
one part.
\end{conjecture}

If this conjecture is proved, then one may consider the problem of
determining $g_r(m,n)$ combinatorially solved: the number of
$K_r$-subgraphs in such a graph $G$ is some explicit polynomial in $m$, $n$,
and part sizes, and the question reduces to its minimization over
the integers. This task may
be difficult but it involves no graph theory. In fact, it is not hard to
show (see e.g.~\cite[Section~3]{nikiforov:11}) that the
optimal part ratios are approximately as those of the graphs in $\C H_a$,
where $a=m/{n\choose 2}$.
(However, our rounding $|V_1|=\floor{cn}$, etc., was rather arbitrary:
it was chosen just to have the family $\C H_a$ well-defined.)

We hope that Theorem~\ref{th:comb} may help in proving
Conjecture~\ref{cj:LS} in the same way as the so-called stability
approach is useful in obtaining exact results. One example where
this approach succeeded is the clique minimization problem in the
special case when
$a=1-\frac1t$ for some integer $t\ge 2$. First, the results of
Nordhaus and Stewart~\cite{nordhaus+stewart:63} (for $r=3$)
and Moon and Moser~\cite{moon+moser:62}
(for $r\ge 4$) imply that for any $m,n$ we have
 \beq{goodman}
 g_r(m,n)\ge \frac{t(t-1)\dots(t-r+1)}{r!}\, \left(\frac nt\right)^r,\quad
\mbox{if $t\ge r-1$},
 \eeq
 where the real $t$ is defined by $m=(1-1/t)n^2/2$. A short proof
can be found in \cite[Problem~10.40]{lovasz:cpe}.
Note that, if $t$ is an integer,
then~\req{goodman} is asymptotically
best possible as shown by
the Tur\'an
graph $T_t(n)$; thus
$g_r(1-\frac1t)=r! {t\choose r}/t^r$ in this case.
Lov\'asz and
Simonovits~\cite[Theorem~2]{lovasz+simonovits:83} deduced that all almost extremal graphs are
close to $T_t(n)$ in the edit distance:

\begin{theorem}\label{th:stab}
For every $r$ and $\e>0$, there are $\delta>0$ and $n_0$ such that, for
any integer $t\ge r-1$, every graph $G$ with $n\ge n_0$ vertices,
$(1-\frac1t\pm\delta) {n\choose 2}$ edges, and at most
$(g_r(1-\frac1t)+\delta) {n\choose r}$ copies of $K_r$ can be made
isomorphic to $T_t(n)$ by changing at most $\e {n\choose 2}$
edges.\end{theorem}

In fact, a sharper form of this result (with an explicit
$\delta=\delta(r,t,\e,n)$) was proved by Lov\'asz and
Simonovits~\cite{erdos+simonovits:83} who used it to establish
Conjecture~\ref{cj:LS} when $\ex(n,K_s)\le m\le \ex(n,K_s)+\e n^2$ for
some $\e=\e(r,s)>0$.

This paper is organized as follows. We outline
the main ideas behind flag algebras and state some of the key
inequalities from \rcpc\ in Section~\ref{flag}. There, we also state our result
on the
structure of $g_3$-extremal homomorphisms (Theorem~\ref{th:Phi}) and
show how this implies Theorem~\ref{th:comb}.  Section~\ref{sketch}
contains a sketch of the proof from~\rcpc\ that $g_3(a)=\hh(a)$.
Theorem~\ref{th:Phi} is proved in
Section~\ref{proof}.

\section{Flag Algebras}\label{flag}

In order to understand this paper the reader should be familiar with
the concepts introduced in~\cite{razborov:07}.
 We do not see any reasonable way
of making this paper self-contained, without making it quite long
and repeating large passages from~\cite{razborov:07}.
Therefore, we restrict ourselves to sketching
the proofs in~\cite{razborov:07,razborov:08}, during which we informally
illustrate the main ideas by providing some analogs from the discrete
world. This serves two purposes: to state the key inequalities
from~\cite{razborov:07,razborov:08} that we need here and to provide some guiding
intuition for the reader who is about to start
reading~\cite{razborov:07}. We stress that some flag
algebra concepts do not have direct combinatorial analogs or require a
plethora of constants to state them in terms of graphs. Here we just try to
distill and present some motivational ideas. Besides, even if the theory was
intentionally developed to cover arbitrary combinatorial structures, in our
brief exposition we confine ourselves to the case of ordinary graphs, as the
most intuitive one.

Many proofs in extremal graph theory proceed by considering possible
densities of small subgraphs and deriving various inequalities between
them. These calculations often become very cumbersome and difficult to
keep track of ``by hand'', especially that the number of
non-isomorphic graphs increases very quickly with the number of
vertices. One of the motivations behind introducing flag algebras was
to develop a framework where the mechanical book-keeping part of the
work is relegated to a computer.

So suppose that we have a graph $G$. Let $n=|V(G)|$ be its order.

The \emph{density} of a graph $F$ in $G$, denoted by $p(F,G)$, is the
probability that a random $|V(F)|$-subset of $V(G)$ spans a subgraph
isomorphic to $F$. The quantities that we are interested in are finite
linear combinations $\sum_{i=1}^s \alpha_i p(F_i,G)$, where $F_i$ is a graph and
$\alpha_i$ is a real constant.  One can view a formal finite sum $\sum_{i=1}^s
\alpha_i F_i$ as a function that evaluates to $\sum_{i=1}^s \alpha_i p(F_i,G)$
on input
$G$. Since we would like to operate with these objects on computers,
we try to keep redundancies to minimum. In particular, the graphs
$F_i$ are unlabeled and pairwise non-isomorphic. Let $\C F^0$ consist
of all (unlabeled non-isomorphic) graphs and let $\I R\C F^0$ be the vector
space that has $\C F^0$ as a basis. (The meaning of the
superscript 0 will be explained a bit later.)

There are some relations which are identically true
when it comes to evaluations on input
$G$: for example if $n\ge \ell\ge |V(\tilde
F)|$ for some graph $\tilde F$ and we know the
densities of
all subgraphs on $\ell$ vertices,
then the density of $\tilde F$ can be
easily determined:
 \beq{PTildeF}
 p(\tilde F,G)=\sum_{F\in \C F^0_\ell} p(\tilde F,F)p(F,G),
 \eeq
 where $\C F^0_\ell\subseteq \C F^0$ consists of all graphs with
exactly $\ell$ vertices.

So it makes sense to factor over $\C K^0$, the subspace of
$\I R\C F^0$ generated by $\tilde F-\sum_{F\in \C F^0_\ell} p(\tilde F,F)F$,
over all choices of $\tilde F$ and $\ell\ge |V(\tilde F)|$. Let
 $$
 \C A^0\df\I R\C F^0/\C K^0.
 $$
 By~\req{PTildeF},
any element of $\C A^0$ can still be identified with an evaluation on
(sufficiently large) graphs.

Let some $F_i\in \C F^0_{\ell_i}$ for $i=1,2$ be fixed. The product
$p(F_1,G)p(F_2,G)$ is the probability that two random subsets
$U_1,U_2\subseteq V(G)$ of sizes $\ell_1$ and $\ell_2$, drawn
independently, induce copies of $F_1$ and $F_2$ respectively. With probability
$1-O(1/n)$ (recall that $n=|V(G)|$), the sets $U_1$ and $U_2$ are
disjoint. Let us condition on this
event. The conditional distribution can be generated
as follows: first pick a random
$(\ell_1+\ell_2)$-set $U$ and then take a random partition $U=U_1\cup
U_2$ with $|U_i|=\ell_i$.
Thus
 \beq{pF1F2}
 p(F_1,G)p(F_2,G)=\sum_{F\in\C F^0_{\ell_1+\ell_2}} p(F_1,F_2;F)p(F,G)+O(1/n),
 \eeq
  where $p(F_1,F_2;F)$ denotes the probability that $F[U_i]\cong F_i$ (i.e.\ the subgraph
of $F$ induced by $U_i$ is isomorphic to $F_i$) for both
$i=1,2$ when we take a random partition
$U_1\cup U_2$ of the vertex set of
$F\in \C F^0_{\ell_1+\ell_2}$ with part sizes $\ell_1$ and $\ell_2$.
 Since we are interested in the case when $n\to\infty$, we
formally define the product $F_1\cdot F_2$ to be equal to
$\sum_{F\in\C F^0_{\ell_1+\ell_2}} p(F_1,F_2;F)\, F\in \I R\C F^0$
and extend this multiplication
to $\I R\C F^0$ by linearity. It is not surprising that this definition is
compatible with
the factorization by $\C K^0$, making $\C A^0$ into a commutative associate algebra with the empty graph being the multiplicative identity,
see~\cite[Lemma~2.4]{razborov:07}.

Unfortunately, we do not have the property that graph evaluations
preserve multiplication exactly.
This can be rectified if we take as input not just a single graph $G$
but a sequence of graphs $\{G_n\}$ which is \emph{convergent} by which
we mean that $|V(G_1)|<|V(G_2)|<\dots$ (we call such sequences
\emph{increasing}) and for every graph $F$ the
limit
 \beq{phi(F)}
 \phi(F)\df \lim_{n\to\infty} p(F,G_n)
 \eeq
 exists. We extend $\phi$ by linearity to $\I R\C F^0$. It
is routine to check that $\phi$ is compatible with the factorization by $\C
K^0$ and, in fact, gives an algebra homomorphism from $\C A^0$ to $\I
R$ (which we still denote by $\phi$), see \rjsl[Theorem~3.3].
We say that $\phi$ is the \emph{limit}
of $\{G_n\}$ and, following the notation in
\rjsl[Section~3.1], denote this as $\phi=\lim_{n\to\infty} p^{G_n}$, where
$p^{G_n}(F)\df p(F, G_n)$ if $|V(F)|\leq |V(G_n)|$ and 0 otherwise.

Clearly, $\phi$ is
\emph{non-negative}, that is, $\phi(F)\ge 0$ for every graph $F$. Let
$\Hom^+(\scr A^0, {\mathbb R})$ be the set of all non-negative homomorphisms.

It turns out that every non-negative homomorphism $\phi:\C A^0\to\I R$ is
the limit of some sequence of graphs. It is instructive to sketch a
proof of this, see Lov\'asz and Szegedy~\cite[Lemma~2.4]{lovasz+szegedy:06}
(or~\cite[Theorem~3.3]{razborov:07} in more general context) for details.
Take some integer
$n$. Since the identity $\sum_{F\in\C F^0_n} F=1$ holds in $\C A^0$,
we have that $\sum_{F\in\C F^0_n} \phi(F)=1$, that is, $\phi$ defines
some probability distribution on $\C F^0_n$. Let $\rn{G_{n,\phi}}\in \C
F^0_n$ be drawn according to this distribution with the choices for
different values of $n$ being independent. Fix some $F$ and $\e>0$.
Let $n\ge |V(F)|$.  An easy calculation shows that the expectation of
$p(F,\rn{G_{n,\phi}})$ is exactly $\phi(F)$. Also, the variance of
$p(F,\rn{G_{n,\phi}})$, which can be expressed via counting pairs of
$F$-subgraphs versus two independent copies of $F$, is
$O(1/n)$. Chebyshev's inequality implies that the probability of the
``bad'' event $|p(F,\rn{G_{n,\phi}})-\phi(F)|>\e$ is $O(1/n)$ and the
Borel-Cantelli Lemma shows that with probability $1$ only finitely
many bad events occur when $n$ runs over, for example, all squares.
Since there are only countably many choices of $F$ and, for example,
$\e\in \{1,\frac12,\frac13,\dots\}$,
we conclude that $\{\rn{G_{n^2,\phi}}\}$ converges to $\phi$
with probability 1.  Thus the required convergent sequence exists.

If one wishes that the graph orders in the sequence span all natural
numbers, one can pick some convergent sequence and fill all
orders by uniformly ``blowing'' up its members, see
e.g.~\cite[Section~2.3]{hatami+hladky+kral+norine+razborov:13}.
Alternatively, one can show that  the sequence $\{\rn{G_{n,\phi}}\}$ itself converges
with probability 1 via a stronger concentration result for
$p(F,\rn{G_{n,\phi}})$ that considers its first four moments, see~\cite[Lemma~11.7]{lovasz:lngl}.

How can these concepts be useful for proving that $g_3(a)=\hh(a)$? Pick an
increasing
sequence of graphs $\{G_{n}\}$ of edge density
$a+o(1)$ such that the limit of $p(K_3,G_n)$ exists and is equal to
$g_3(a)$. A standard diagonalization argument
shows that $\{G_n\}$ has a convergent subsequence; let $\phi$ be its limit.
Then
$\phi(K_2)=a$. Now, if we can show that
 \beq{LowerBound}
 \forall\, \phi\in \Hom^+(\scr A^0, {\mathbb R})\quad \left(\phi(K_2)=a\quad
\Longrightarrow\quad \phi(K_3)\ge \hh(a)\right),
 \eeq
 then we can conclude that indeed $g_3(a)=\hh(a)$, as it was done in \rcpc.

In this paper, we achieve more: we describe the set of all
\emph{extremal} homomorphisms, that is, those $\phi\in \Hom^+(\scr A^0, {\mathbb R})$
that achieve equality $\phi(K_3)=g_3(\phi(K_2))$.

Let $\Phi\subseteq
\Hom^+(\scr A^0, {\mathbb R})$ consist
of all possible limits of convergent sequences $\{G_n\}$ for which
there is $a\in[0,1]$  such that $G_n\in \C H_a$ for all $n$.
Equivalently, $\Phi$ can be defined as follows.
Recall that the \emph{join} $G_1\vee \ldots \vee G_k$ of graphs
$G_1,\ldots, G_k$ is
obtained by taking their disjoint union and adding all edges in
between. We define a similar operation on homomorphisms
$\phi_1,\dots,\phi_k\in \Hom^+(\scr A^0,
{\mathbb R})$. We need a more general construction where one specifies
how much relative weight each $\phi_i$ has, by giving
non-negative reals $\alpha_1,\dots,\alpha_k$ with sum $1$. Let
$n\to\infty$ and, for $i\in [k]$,
let $G_{i,n}$ be a graph with $\floor{\alpha_i n}$ vertices such that
the sequence $\{G_{i,n}\}$ converges to $\phi_i$; as we
have already remarked, it exists.
Let $F_n=G_{1,n}\vee\dots\vee G_{k,n}$. Let  the \emph{join}
$\phi=\vee(\phi_1,\dots,\phi_k;\alpha_1,\dots,\alpha_k)$
be the limit of $\{F_n\}$ (it is easy to see that the limit exists).

Alternatively, we can define the join $\phi$ without appealing to
convergence. To this end, it is enough to define the density of each graph
$F\in\C F^0$, and we do it as follows. Let $\mathrm{aut}(F)$ denote
the number of automorphisms of $F$. Let
 \begin{equation}\label{DefJoin}
  \phi(F)\df \frac{|V(F)|!}{\mathrm{aut}(F)}\sum_{(V_1,\dots,V_k)} \prod_{i=1}^k\left(
\alpha_i^{|V_i|}\,\phi_i(F[V_i])\,\frac{\mathrm{aut}(F_i)}{|V_i|!}\right),
 \end{equation}
 where the summation runs over all possible ways (up to isomorphism) to
partition
$V(F)=V_1\cup\dots\cup V_k$ into $k$ labeled parts (allowing empty parts)
so that the induced bipartite subgraph $F[V_i,V_j]$ is complete for all $1\le
i<j\le k$. The reader is welcome to formally check
that the join
is well-defined (with respect to the factorization by $\C K^0$)
and belongs to $\Hom^+(\scr A^0, {\mathbb R})$. (These facts are obvious
from the first definition.) Now, $\Phi$ is exactly the set
of all possible joins
 $$
 \vee(\underbrace{0,\dots,0}_{t-1 \mbox{ \scriptsize times}},\psi;
\underbrace{c,\dots,c}_{t-1 \mbox{ \scriptsize times}},1-(t-1)c),
 $$
 where
 $0$ denotes the (unique) non-negative homomorphism in
$\Hom^+(\scr A^0, {\mathbb R})$ of
zero edge-density, $\psi\in \Hom^+(\scr A^0, {\mathbb R})$ is arbitrary with
$\psi(K_3)=0$ and $\psi(K_2)=2c(1-tc)/(1-(t-1)c)^2$, and
$c$ is a real from the interval $[1/(t+1),1/t)$.

Our main result states that the set of $g_3$-extremal homomorphisms is
exactly $\Phi$.

\begin{theorem}\label{th:Phi}
 $$
 \Phi=\left\{\phi\in\Hom^+(\scr A^0, {\mathbb R})\mid \phi(K_3)=
g_3(\phi(K_2))\right\}.
 $$
\end{theorem}

Let us show that Theorem~\ref{th:Phi} implies
Theorem~\ref{th:comb}. The shortest way is to refer to some known
results about the so-called \emph{cut-distance} $\delta_\Box$ that goes back to
Frieze and
Kannan~\cite{frieze+kannan:99}. We omit the definition of $\delta_\Box$ but refer
the reader to~\cite[Definition~2.2]{BCLSV:08} (see also \cite[Chapter~8]{lovasz:lngl}).

Suppose for the sake of contradiction that Theorem~\ref{th:comb} is false,
which is witnessed by some $\e>0$. Then  we can find an increasing
sequence $\{G_n\}$ of graphs
with $p(K_3,G_n)\le g_3(p(K_2,G_n))+o(1)$ that violates the conclusion
of  Theorem~\ref{th:comb}. By passing to a subsequence, we can assume that
$\{G_n\}$ is convergent.
 Let $\phi_0\in \Hom^+(\scr A^0, {\mathbb R})$ be its limit. Let $a=\phi_0(K_2)$.
Clearly, $\phi_0(K_3)=g_3(a)$. By Theorem~\ref{th:Phi},
$\phi_0\in\Phi$ and we can choose a sequence $\{H_n\}$ in $\C H$
which converges to $\phi_0$ with $V(H_n)=V(G_n)$.

This convergence means that asymptotically $G_n$ and $H_n$ have the same statistics
of fixed subgraphs. This does not necessarily implies that $G_n$ and $H_n$ are
close in the edit distance. (For example, two typical random graphs of
edge density $1/2$ have similar subgraph statistics but are far
in the edit distance.) However, the presence of a spanning complete partite
graph in $H_n$ implies a similar conclusion about $G_n$ as follows.

Theorem~2.7 in Borgs et al~\cite{BCLSV:08} gives that
$\delta_\Box(G_n,H_n)=o(1)$, that is, the cut-distance between $G_n$
and $H_n$ tends to 0. (An important property of the cut-distance
is that an increasing sequence $\{G_n\}$ is convergent if and only if it is
Cauchy
with respect to $\delta_\Box$.)

By~\cite[Theorem~2.3]{BCLSV:08}, we can relabel
$V(H_n)$ so that for every disjoint $S,T\subseteq V(G_n)$ we have
 \beq{ST}
 \left|e(G_n[S,T])-e(H_n[S,T])\right|=o(v^2),
 \eeq
 where $v=v(n)$ is the number of vertices in $G_n$.
Informally, this means that the graphs $G_n$ and $H_n$ have almost the same edge
distribution with respect to cuts. Take the partition
$V(H_n)=V_1\cup\dots \cup V_{t-1}\cup U$ that was used to define
$H_n$. Let $i\in[t-1]$. If we set $S=V_i$ and $T=V(G_n)\setminus V_i$
in~\req{ST}, then we conclude that the number of $S-T$ edges that are missing from
$G_n$ is $o(v^2)$. Also, the number of edges in
$G[V_i]$ is $o(v^2)$ for otherwise a random partition $V_i=S\cup T$
would contradict~\req{ST}. Thus, by changing $o(v^2)$ adjacencies in $G_n$, we
can assume that the graphs
$G_n$ and $H_n$ coincide except for the subgraph induced by $U$.
Suppose that $|U|=\Omega(n)$ for otherwise
we are done. We have
 $$
 |e(G_n[U])-e(H_n[U])|= |e(G_n)-e(H_n)|=o(v^2).
 $$
 Of course, when we modify $o(v^2)$ adjacencies in $G_n$, then the number of
triangles
changes by $o(v^3)$. Each
edge of $G_n[U]$ (and of $H_n[U]$)
is in the same number of triangles with the third
vertex belonging to $V(G_n)\setminus U$. Since $H_n[U]$ is
triangle-free and $G_n$ is asymptotically extremal, we
conclude that $G_n[U]$ spans $o(v^3)$ triangles. By the Removal
Lemma~\cite{ruzsa+szemeredi:78,erdos+frankl+rodl:86} (see
e.g.~\cite[Theorem~2.9]{komlos+simonovits:96}), we can make $G_n[U]$
triangle-free by deleting $o(v^2)$ edges.

If $e(G_n[U])\ge e(H_n[U])$, then we just remove some
edges from $G_n[U]$ until exactly $e(H_n[U])$ edges are left, in which case the obtained graph
$G_n$ belongs to $\C H_{a,n}$ and Theorem~\ref{th:comb} is proved.
Otherwise we obtain the same conclusion for all large $n$
by applying the following lemma to $G_n[U]$ and $s=e(H_n[U])$.

\begin{lemma}\label{lm:UpK3Free} For every $\e>0$ there are $\delta>0$ and $n_0$ such
that for every $K_3$-free graph $G$ on $n\ge n_0$ vertices and every integer $s$
with \beq{m}
 e(G)<s\le\min\left(e(G)+\delta n^2,\floor{n^2/4}\right)
 \eeq
 one can change at most
$\e n^2$ adjacencies in $G$ so that the new graph is still $K_3$-free
and has exactly $s$ edges.\end{lemma}

\begin{proof}
 \newcommand{\ccc}{c}
Clearly, it is enough to show how to ensure \emph{at
least} $s$ edges in the final $K_3$-free graph. Given $\e>0$,
choose small positive constants $\ccc\gg \delta$.
Let $n$ be
large and let $s$ satisfy~\req{m}.
Let $m=e(G)$.

We can assume that, for example, $m\ge \e n^2/3$. Also,
assume that $m\le \floor{n^2/4}-\ccc
n^2$ for otherwise
we are done by the Stability Theorem of
Erd\H{o}s \cite{erdos:67a} and Simonovits \cite{simonovits:68} which implies
that $G$ can be transformed into the Tur\'an graph $T_2(n)$
by changing at most $\e n^2$ adjacencies.

The number $p$ of paths of length $2$ in $G$ is
$\sum_{x\in V(G)} {d(x)\choose 2}$ which is at least $n{2m/n\choose 2}$
by the convexity of the function ${x\choose 2}$. By averaging, there is an
edge $xy\in E(G)$ that belongs to at least
 $$
 \frac{2p}m \ge \frac{2n{2m/n\choose 2}}m \ge \frac{4m}n - \delta n
 $$
 such paths (which is just the number of edges between the set
$\{x,y\}$ and its complement).

Let $G'$ be obtained from $G$ by adding $\ccc n$ clones of $x$ and
$\ccc n$ clones of $y$. Thus $G'$ has $n'=(1+2\ccc )n$ vertices and $m'\ge
m+\ccc n(\frac{4m}n - \delta n)+(\ccc n)^2$ edges. If we take a random
$n$-subset $U$ of $V(G')$, then each edge of $G'$ is included with probability
${n\choose 2}/{n'\choose 2}$. Thus there is a choice of an $n$-set $U$
such that the number of edges in $H=G'[U]$ is at least the average,
which in turn is at least
 $$
 \frac{\left(m+\ccc n(\frac{4m}n - \delta n)+(\ccc n)^2\right){n\choose 2}}
{{(1+2\ccc )n\choose 2}}\ge m+
\frac{\ccc ^2(n^2-4m) - 2\ccc \delta n^2}{(1+2\ccc )^2}.
 $$
 This is at least $m+\delta n^2\ge s$ by our assumption on $m$. Since
$G$ and $H$ coincide on the set $V(G)\cap V(H)$ of
least $n-2\ccc n$ vertices,  $G$
can be transformed into the $K_3$-free graph
$H$ by changing at most $2\ccc n^2\le \e n^2$
adjacencies, as required.\end{proof}

\section{Sketch of Proof of $\phi(K_3)\ge \hh(\phi(K_2))$}\label{sketch}

Let us sketch the proof of~\req{LowerBound} from
\cite{razborov:07,razborov:08}, being consistent
with the notation defined there. Let $\rho\df K_2\in \scr F^0_2$. Consider
the ``defect'' functional $f(\phi)=\phi(K_3)-\hh(\phi(\rho))$, where
$\hh$  is defined by~\req{gr}.
We can identify each homomorphism $\phi\in\Hom(\C A^0,\I R)$ with the sequence
 $$
 (\phi(F))_{F\in\C F^0}\in  \I R^{\C F^0}
 $$
 of its values on graphs. Let us equip all products
with the pointwise convergence (or product) topology.
The set $\Hom(\C A^0,\I R)$
is a closed subset of $\I R^{\C F^0}$ as the intersection
of closed subsets corresponding to the relations that an algebra
homomorphism has to satisfy. Thus the set
 $$
 \Hom^+(\scr A^0, {\mathbb R})=\bigcap_{F\in\C F^0}
\left\{\phi\in\Hom(\C A^0,\I R)\mid \phi(F)\ge 0\right\}
 $$
 is closed too. Moreover, it lies inside the compact space $[0,1]^{\C F^0}$,
so it is compact as well.
Since $\hh(x)$ is a continuous function (including the special point
$x=1$),
our functional $f$ is also continuous and
achieves its smallest value on $\Hom^+(\scr A^0, {\mathbb R})$ at some
non-negative homomorphism $\phi_0$. Fix one such $\phi_0$ for the rest of
the proof. Let $a=\phi_0(\rho)$. Let
$t=t(a)$ and $c=c(a)$ be defined as in the Introduction. Let $b=\phi_0(K_3)$.
We have to show that $b\ge \hh(a)$.

If $a=\phi(\rho)\le 1/2$, then $h(a)=0$ and there is nothing to do.

Let us write an explicit formula for the function
$\hh(x)$ defined in~\req{gr} when $1-\frac1t\le x\le 1-\frac1{t+1}$:
 \begin{equation}\label{ht}
  h_t(x)\df \frac{(t-1)\left(t-2\sqrt{t(t-x(t+1))}\right)\left(t+\sqrt{t(t-x(t+1))}\right)^2}{t^2(t+1)^2}.
 \end{equation}

If $a=1-\frac1{t+1}$, then we are done by~\req{goodman}.
So let us assume that $a$ lies in the open interval $(1-\frac1t,1-\frac1{t+1})$. Here
the function $h_t(x)$ is differentiable and it is routine to see that
$h_t'(a)=3(t-1)c$.
A calculation-free intuition  is that
if we add one edge to $H\in \C H_a$ then the
number of triangles increases by $((t-1)c+o(1))n$ (while the effect of
the change in the part sizes is relatively negligible); so we expect that
$h_t'(a) {n\choose 2}^{-1}\approx (t-1)cn {n\choose 3}^{-1}$.
 \hide{When we increase
$a$ by ${n\choose 2}^{-1}$, $c=c(a)$ also changes but only by $O(n^{-2})$; so
the effect of this `change' on the triangle density is negligible.
 }

Let us see which properties $\phi_0$ has. Let $\{G_n\}$
converge to $\phi_0$ with $|V(G_n)|=n$. Let $\e>0$ be a
small constant.

It is impossible that at least $\e n^2$ edges of $G_n$ are each in
more than $((t-1)c+\e) n$ triangles: by removing a uniformly spread
subset of these edges we get a change that is noticeable in the limit
and strictly decreases the defect functional $f$. Thus, if we pick a random
edge from $E(G_n)$, then with probability $1-o(1)$ there are
at most $((t-1)c+o(1)) n$ triangles containing this edge. (Note that
$G_n$ has $\Omega(n^2)$ edges by our assumption $a\ge 1/2$.)
The corresponding  flag algebra
statement \rcpc[(3.3)] reads
 \beq{3.3}
 \rn{\phi_0^E}(K_3^E)\le \frac13 h_t'(a)\quad \mbox{a.e. (=almost everywhere)}.
 \eeq

Let us informally explain~\req{3.3}. It involves counting triangles that
contain a specified edge. Let $\C F^E$
consist of \emph{$E$-flags}, by which we mean
graphs with some two adjacent vertices being labeled as $1$ and $2$.
Any isomorphism has to preserve the labels. We
may represent elements of $\C F^E$ as $(G;x_1,x_2)$,
where $G\in\C F^0$ is a graph and $x_i\in V(G)$
is the vertex that gets label~$i$.
Suppose that we wish to keep track of various subgraph densities and their
finite linear combinations for $E$-flags.
We can view $(F;y_1,y_2)\in\C F^E$ as an evaluation on $\C F^E$ that
on input
$(G;x_1,x_2)$ returns $p((F;y_1,y_2),(G;x_1,x_2))$, the probability
that the $E$-subflag of $G$ induced by
a random $|V(F)|$-set $X$ with $\{x_1,x_2\}\subseteq X
\subseteq V(G)$ is isomorphic to $(F;y_1,y_2)$.

Again, if we know the densities of all $E$-flags with $\ell\ge |V(F)|$
vertices, then we can determine the density of $(F;y_1,y_2)$ by the analog
of~\req{PTildeF}.
So we can define the corresponding linear subspace $\C K^E$ and let
$\C A^E\df \I R\C F^E/\C K^E$. The obvious analog of~\req{pF1F2}
holds, and the corresponding coefficients define a multiplication
on $\I R\C F^E$ that turns $\C A^E$ into a commutative algebra. The
multiplicative identity is $E\in \C F^E$, the unique $E$-flag on $K_2$.
As in the
unlabeled case, the limits of convergent sequences of $E$-flags are precisely
non-negative algebra homomorphisms from $\C A^E$ to the reals
(\rjsl[Theorem~3.3]).

Now, we can turn $G_n$ into an $E$-flag by taking a random edge
uniformly from $E(G_n)$ and randomly labeling its endpoints by $1$ and
$2$. Thus for each $n$ we have a probability distribution on $E$-flags
which weakly converges to the distribution on
$\Hom^+(\C A^E,\I R)$, and it is very important that this distribution can
be uniquely retrieved from $\phi_0$ only (see \rcpc[Section~3.2]). In particular,
it will not depend on the choice of the representing convergent sequence $\{G_n\}$.
In~\req{3.3},  $\rn{\phi_0^E}$
denotes the \emph{extension} of $\phi_0$ (that is,
a random homomorphism drawn according to this distribution)
while $K_3^E$
is the unique $E$-flag with the underlying graph being $K_3$.

Let us consider the effect of removing a vertex $x$ from $G_n$. When
we first remove $d(x)$ edges at $x$, the edge density goes down by
$d(x)/{n\choose 2}$. Next, when we remove the (now isolated) vertex $x$,
the edge density is multiplied by ${n\choose 2}/{n-1\choose
2}=1+\frac2n+O(n^{-2})$. Thus the edge density
changes by $-d(x)/{n\choose 2}+2a/n+O(n^{-2})$.
Likewise, the triangle density
changes by $-K_3^1(x)/{n\choose 3}+3b/n+O(n^{-2})$,
where $K_3^1(x)$ is the number
of triangles per $x$. Thus for all but at most $\e n$ vertices $x$ we
have $(-2d(x)/n+2a)h_t'(a) < -3K_3^1(x)/{n\choose 2}+3b+\e$,
for otherwise by removing $\e n$
such vertices (and taking the limit of a convergent subsequence of the
resulting graphs)
we can strictly decrease the defect functional $f$.
In the flag algebra language this reads as
 \begin{equation}\label{3.2}
 -2h_t'(a)\rn{\phi_0^1}(K_2^1)+2h_t'(a)a\le -3\rn{\phi_0^1}(K_3^1)+3b,
\quad \mbox{a.e.,}
 \end{equation}
 where $\C F^1$ consists of all graphs with one vertex labeled $1$,
$K_2^1,K_3^1\in \C F^1$ ``evaluate'' the edge and triangle density at
the labeled vertex, and  $\rn{\phi_0^1}\in\Hom^+(\C A^1,\I R)$
is the random extension of~$\phi_0$ constructed similarly\footnote{Now
it is an appropriate place to observe that the
superscript in $\C F^0$ refers to the empty type $0$.} to $\rn{\phi_0^E}$.

Note that if we take the expectation of each side of~\req{3.2} with
respect to the random $\rn{\phi_0^1}\in\Hom^+(\C A^1,\I R)$, then we
get $0$. (A calculation-free intuition is that
the edge/triangle density of a graph $G$ is equal to the average density of
edges/triangles sitting on a random vertex of $G$.) Thus we
conclude that~\req{3.2} is in fact equality a.e.\ (\rcpc[(3.2)]).

How can~\req{3.3} and~\req{3.2} be converted into statements about
$\phi_0$? If, for example, one applies the averaging operator $\eval{...}1$  (\cite[Section~2.2]{razborov:07}) to~\req{3.2}, that is, taking the expected value
of~\req{3.2} over $\rn{\phi_0^1}$, then one obtains the identity $0=0$,
as
we have just mentioned. However, one
can  multiply both sides of~\req{3.2} by some $1$-flag $F$ and then average.
(In terms of graphs this corresponds to weighting vertices of $G_n$
proportionally to the density of
$F$-subgraphs rooted at them.) What sufficed in
\cite{razborov:07,razborov:08} was to take $F=K_2^1$. Denoting $e=K_2^1$
for convenience and rearranging terms, we get (\rcpc[(3.4)]):
 \beq{3.4}
 \phi_0(3\eval{eK_3^1}1-2h_t'(a)\eval{e^2}1)=a(3b-2ah_t'(a)).
 \eeq

\hide{
Theorem~3.18 in~\rcpc shows the intuitively obvious fact that the
averaging operator preserves inequalities.
}
Applying the operator $\eval{\dots}E$ (averaging over $\rn{\phi_0^E}$)
directly
to~\req{3.3} is not useful.
Namely, if we take a graph $G\in \C H_a$,
then the graph analog of~\req{3.3} may have slack for edges that connect
two larger parts; thus the obtained inequality will not be best possible.
The trick in~\cite{razborov:07} was first to multiply~\req{3.3}
by the $E$-flag $\bar P_3^E$ whose graph is
the complement of the 3-vertex path. (Thus each edge of $\C H_a$
with slack gets weight 0.) We obtain (\rcpc[(3.5)]):
 \beq{3.5}
 \phi_0(\eval{\bar P_3^EK_3^E}E)\le \frac13 h_t'(a)\phi_0(\eval{\bar P_3^E}E)=\frac19 h_t'(a)\phi_0(\bar P_3).
 \eeq

We will also need the following identity which may be routinely
checked (compare with \rcpc[Lemma~3.2]):
 \beq{Lm3.2}
 3\eval{eK_3^1}1+3\eval{\bar P_3^EK_3^E}E=2K_3+K_4+\frac14\, \bar K_{1,3},
 \eeq
 where $K_{s,t}$ is the complete bipartite graph with part sizes $s$ and $t$.
(Thus $\bar K_{1,3}$ is a triangle plus an isolated vertex.) Also, we have
 \beq{id2}
 \frac13\bar P_3+2\eval{e^2}1=\rho+K_3.
 \eeq
 Now, if we apply
$\phi_0$ to~\req{Lm3.2} and~\req{id2} and combine with
\req{3.4} and  \req{3.5}, then we obtain the following
inequality (see\ \rcpc[(3.6)] where it is also proved that $h_t'(a)+3a-2>0$):
 \beq{3.6}
 b\ge \frac{a(2a-1)h_t'(a)+\phi_0(K_4)+\frac14\, \phi_0(\bar K_{1,3})}
{h_t'(a)+3a-2}.
 \eeq
 If $\phi_0(\bar K_{1,3})=0$ and $\phi_0(K_4)$ is equal to the
limiting $K_4$-density in $\C H_a$, then the right-hand side of~\req{3.6}
is exactly $\hh(a)$. Thus it remains to bound $\phi_0(K_4)$ from below.
In particular,
we are already done if $a\le 2/3$ since every graph in
$\C H_a$ has no (or very few) copies of $K_4$; this is
what was done in \cite{razborov:07}. Of course,
the result of Nikiforov~\cite{nikiforov:11} who determined $g_4(a)$ for all $a$
would suffice
here but in order to
prove our new Theorem~\ref{th:Phi} we need to analyze the argument of \rcpc\ further.

Following \cite[page~612]{razborov:08} define
 \begin{eqnarray}
 \nonumber A&\df & \frac23 h_t'(a)\ =\ 2(t-1)c,\\
 \label{b_original} B&\df&Aa-b\ =\ \frac23 ah_t'(a)-b.
 \end{eqnarray}
 Then, for example, \req{3.2}, which is an equality a.e., can be rewritten as
 \beq{3.9}
 \rn{\phi_0^1}(K_3^1)=A\rn{\phi_0^1}(e)-B\quad \mbox{a.e.}
 \eeq

 Also, let us apply the averaging operator
$\eval{\dots}{E,1}$ to~\req{3.3}.
Informally speaking, given the labeled vertex
$x_1\in V(G_n)$, we pick the second labeled
vertex $x_2$ uniformly at random and take the expectation of~\req{3.3}
multiplied by the indicator function of $x_1$ and $x_2$ being adjacent.
\hide{Note that
 $$
 B=\frac23 ah_t'(a)-b\ge 2(t-1)ac-\hh(a)=t(t-1)c^2>0,
 $$
 so by~\req{3.9} we have $\rn{\phi_0^1}(e)>0$ a.e.
}
Since $\eval{K_3^E}{E,1}=K_3^1$ and $\eval{1}{E,1}=\eval{E}{E,1}=e$, we get (\rcpc[(3.8)])
 \beq{3.8}
 \rn{\phi_0^1}(K_3^1)\le \frac13 h_t'(a)\,\rn{\phi_0^1}(e)=\frac A2\,\rn{\phi_0^1}(e)\quad \mbox{a.e.}
 \eeq
 The combinatorial meaning of the last step is very simple: if each edge is in at most $(t-1)cn$ triangles, then a given vertex $x_1$ can belong to at most $\frac12 d(x_1)(t-1)cn$ triangles.

From~\req{3.9} and~\req{3.8} we obtain
 \beq{3.11}
 0<\frac BA\le \rn{\phi_0^1}(e)\le \frac{2B}A\quad \mbox{a.e.}
 \eeq

Now let us take any \emph{individual} $\phi^1\in \Hom^+(\C A^1,\I R)$ for
which~\req{3.9}--\req{3.11} hold. Let
 \beq{psi}
 \psi\df \phi^1\pi^e\in
\Hom^+(\C A^0,\I R),
 \eeq
 see \rcpc[page~612].
 Informally, we take an arbitrary vertex $x$ of $G_n$
and assume that the density of edges/triangles containing $x$
satisfies~\req{3.9}--\req{3.11}. Then $\psi$ corresponds to taking
the subgraph $H_n$ of $G_n$ induced by the neighborhood of $x$. For
example, the edge density of $H_n$ can be calculated by taking
the triangle density at $x$ and multiplying it by
${n-1\choose 2}/{d(x)\choose 2}\approx (\frac{n-1}{d(x)})^2$. In the flag
algebra formalism this reads (\rcpc[(3.13)])
 \beq{3.13}
 \psi(\rho)=\frac{\phi^1(K_3^1)}{(\phi^1(e))^2}=\frac{A\phi^1(e)-B}{(\phi^1(e))^2}
=\frac{z-\mu}{z^2},
 \eeq
 where following \rcpc[page~612] we define
 \beq{zmu}
 z\df\phi^1(e)/A\quad\mbox{and}\quad \mu\df B/A^2.
 \eeq
 Some
calculations based on~\req{goodman} show that (\rcpc[(3.15)])
\beq{3.15}
 \psi(\rho)\le 1-\frac1t.
 \eeq

Summarizing (in the graph theory language): the degree of
a typical $x\in V(G_n)$ determines the edge density of $G_n[N(x)]$,
the subgraph induced by the
neighborhood $N(x)$ of $x$.
Moreover, this density is at most $1-\frac1t+o(1)$. This
give us a strategy for bounding the number of $K_4$'s in $G_n$ from below:
use induction on $t$ to bound the number of $K_3$'s in
$N(x)$ and then sum this over all $x\in V(G_n)$ (and divide
by~4). Unfortunately, this bound on $\psi(K_3)$
involves radicals and it is not clear how to average it, since
$t(\psi(\rho))$ may assume different values for different
choices of $\phi^1$. These difficulties are overcome
by proving
the following lower bound on $\phi^1(K_4^1)=\psi(K_3)
(\phi^1(e))^3$ which
is a linear function of $\phi^1(e)$ that does not depend on $t(\psi(\rho))$
(\rcpc[(3.24)]):
 \beq{3.24}
 \phi^1(K_4^1)\ge A^3\left(\frac32(1-2\mu)\left(\frac{\phi^1(e)}A-\eta_{t-1}\right)
+ \eta_{t-1}^3 \, \frac{(t-2)(t-3)}{(t-1)^2}\right),
 \eeq
 where, for $1\le s\le t-1$, $\eta_s$ is the unique root of the equation
 \beq{3.17}
 \frac{\eta_s-\mu}{\eta_s^2}=1-\frac1s
 \eeq
 that lies in the interval $[\mu,2\mu]$, see \rcpc[(3.17)].
Thus the random extension $\rn{\phi_0^1}$ satisfies~\req{3.24} a.e.\
and we can average it, obtaining a lower bound on $\phi_0(K_4)$, which
is \rcpc[(3.25)]. (Note that
the expectation of $\rn{\phi_0^1}(K_4^1)$ is
$\phi_0(K_4)$.) It turns out that this lower bound, when substituted
into~\req{3.6} suffices for proving the desired conclusion $b\ge \hh(a)$. The
derivations (also those of~\req{3.24}) are rather messy,
do not involve any genuine flag algebras calculations and are not needed
for our proof. So we omit them
and refer the reader to \rcpc\ for all details.

\section{Proof of Theorem~\ref{th:Phi}}\label{proof}

All notation here is compatible with that of \cite{razborov:07,razborov:08}.
As before, let $0$, $1$, and $E$
denote
the (unique) types with respectively 0, 1 and 2 (adjacent) vertices. Also,
$\rho\df K_2\in\C F_2^0$ and $e\df K_2^1\in\C F^1_2$ are the (unique)
$0$- and $1$-flags having two adjacent vertices. In the arXiv version of
our paper (\texttt{arXiv.org:1204.2846}) we offer a \emph{Mathematica} code that verifies some laborious flag
algebra (in)equalities that are needed here.

Let $\Phi\subseteq \Hom^+(\scr A^0, {\mathbb R})$ be
the set of the conjectured extremal homomorphisms defined in Section~\ref{flag}.
Let $\phi_0\in
\Hom^+(\scr A^0, {\mathbb R})$ be arbitrary such that
$\phi_0(K_3)=\hh(\phi_0(\rho))$. We have to show that
$\phi_0\in\Phi$. Let $a\df \phi_0(\rho)$ and $b\df \phi_0(K_3)$.

We prove Theorem~\ref{th:Phi} (that is, the claim that $\phi_0\in\Phi$)
by induction on the parameter $t=t(a)$ that was defined by~\eqref{t}. If $t=1$,
then $a\le 1/2$,
$b=0$, and there is nothing to do: every
non-negative homomorphism of triangle density 0 is in $\Phi$ by
definition.
Let $t\ge 2$ and
assume that we have proved the theorem for all smaller $t$.

Suppose first that $a=1-\frac1s$ for some integer $s$. Apply
Theorem~\ref{th:stab} to any sequence $\{G_n\}$ convergent to
$\phi_0$, say with $|V(G_n)|=n$, to conclude that $G_n$ is $o(n^2)$-close to the Tur\'an graph $T_s(n)$ in the edit distance. Clearly,
when we change $o(n^2)$ edges in $G_n$, then the density of any fixed
graph $F$ changes by $o(1)$ so $\phi_0$ is still the limit of
$\{G_n\}$. Since the limit of $\{T_s(n)\}$ is in
$\Phi$, we are done in this case.

So let $a$ lie in the open interval $(1-\frac 1t, 1-\frac
1{t+1})$. Let $c$ be defined by~\req{CExplicit}.
We assume that the reader is familiar with
the proof in  \cite{razborov:08}; part of it was
sketched in Section~\ref{sketch}, and
we utilize the notation and facts established
there.

Since $\phi_0$ is extremal,
we know that $b=\hh(a)$. This gives some noticeable simplifications
to \eqref{b_original}, \req{zmu} and \req{3.17}:
\begin{eqnarray}
 B &=& t(t-1)c^2, \nonumber\\
\mu &=& \frac{B}{A^2}\ =\ \frac t{4(t-1)},\label{mu}\\
\eta_{t-1} &=& 1/2\nonumber.
\end{eqnarray}

The \emph{support} of the random extension $\rn{\phi_0^\sigma}$ discussed in the
previous section is
the smallest closed subset of
$\Hom^+(\C A^\sigma,\I R)$ of measure $1$; it will be denoted by
$S^\sigma(\phi_0)$. A useful property of the support
is that if some closed property has measure 1, then \emph{every} element
of $S^\sigma(\phi_0)$ has this property.
We fix an arbitrary $\phi^1\in S^1(\phi_0)$. Inequalities~\req{3.9}--\req{3.11}
hold a.e.\ and define a closed subset, thus $\phi^1$ satisfies them. In particular, \req{3.11}
simplifies to
 \beq{3.11simple}
 0<\frac{tc}2\le \phi^1(e)\le tc<1.
 \eeq
So, we can define  $\psi$ by~\req{psi}.

Let us prove that $\psi$ is extremal (that is, has the smallest
possible triangle density given its edge density). It is this part
of our proof that most heavily relies upon \rcpc; it basically amounts to
checking that the extremality assumption $b=h(a)$ makes tight sufficiently
many useful inequalities proven there.

\begin{claim} \label{inductive}
$\psi\in\Phi$ and $\psi(\rho)\in \left[1-\frac 1{t-1}, 1-\frac 1t\right]$. \end{claim}
\begin{proof} Let $s$ be such that $\psi(\rho)\in (1-\frac1{s},1-\frac1{s+1}]$.

We know that the result of averaging \req{3.24} (which is \rcpc[(3.25)]) is an
equality. Hence \req{3.24} is equality a.e., and by the same token as before,
it holds for every $\phi^1\in S^1(\phi_0)$. The analysis
of the calculations in \rcpc\ shows that \rcpc[(3.16)] (which is
equivalent to $\psi(K_3)\ge
h_s(\psi(\rho))$) is also equality.
Thus the homomorphism $\psi\in \Hom^+(\scr A^0,
{\mathbb R})$ is extremal. By \req{3.15}
we have that $s\le t-1$.
The (global) induction assumption implies that $\psi\in\Phi$.

We still have to show the second part of the claim when $t\ge 3$.
Recall that $\psi(\rho)=\frac{z-\mu}{z^2}$
by~\req{3.13}. In view of~\req{mu}, the quadratic equation $\frac{z-\mu}{z^2}=1-\frac1{t-1}$
has two roots:
$z=\frac12$ and $z=\frac{t}{2(t-2)}$.
By (\ref{3.11simple}),
it is impossible that $z\ge \frac{t}{2(t-2)}$ (which is equivalent to $\phi^1(e)\ge
\frac{t(t-1)}{t-2}\,c$). Thus, if we assume that $s\le t-2$, then $\psi(\rho)\le
1-\frac1{t-1}$ and
$z\le 1/2=\eta_{t-1}$.

Thus, when we apply the proof of \rcpc[Claim~3.3], the case $z\le \eta_{t-1}$
takes place.  This implies that
\cite[(3.21)]{razborov:08} is tight.
Then \rcpc[(3.23)] is also tight. Its proof on page 615 of \rcpc\ shows that
this is possible
only if $\mu=\frac{s+1}{4s}$ is the largest element of
$[\frac z2,\frac{s+1}{4s}]$, the admissible interval for $\mu$.
By~(\ref{mu}) we have that $s=t-1$, as required.
\end{proof}

Claim \ref{inductive} alone suffices to verify Theorem \ref{th:Phi} in the toy-like
case $\phi_0(\bar P_3)=0$, where $\bar P_3$ denotes the complement of the 3-vertex path;
combinatorially this means that $\phi_0$ is the limit of complete multipartite graphs.
Indeed, $\phi_0(\bar P_3)=0$ obviously implies that the homomorphism $\psi$ defined by
\req{psi} also satisfies $\psi(\bar P_3)=0$ and, moreover, $\phi_0$ is equal to the join
$\vee(0,\psi; 1-\phi^1(e), \phi^1(e))$. The latter fact readily follows from definitions;
combinatorially it means that every vertex $x$ in a complete multipartite graph $G_n$ defines
its decomposition as the join $G_n= I_n\vee H_n$, where $H_n$ is the subgraph induced by all
neighbors of $x$ and $I_n$ is the independent set induced by all non-neighbors. Thus, applying
Claim \ref{inductive} inductively, we conclude that every $\phi_0\in\Phi$ with $\phi_0(\bar P_3)=0$
necessarily has the form  $\vee(\underbrace{0,\dots,0}_{k \mbox{ \scriptsize times}};
c_1,\ldots, c_k)$, where, say, $0<c_1\leq \ldots\leq c_k$, for some {\em fixed finite $k$}. We are
only left to prove that $c_2=\ldots=c_k$, and the simplest way of doing this is to invoke
\cite[Claim 2.13]{nikiforov:11} used by Nikiforov for essentially identical purpose:
\begin{claim}
Let $\gamma_3\geq \gamma_2\geq \gamma_1>0$ be real numbers satisfying
\begin{eqnarray*}
\gamma_1+\gamma_2+\gamma_3 &=& \alpha,\\
\gamma_1\gamma_2+\gamma_2\gamma_3+\gamma_3\gamma_1 &=& \beta,
\end{eqnarray*}
and let $\gamma_1\gamma_2\gamma_3$ be minimized subject to these two constraints. Then $\gamma_2=\gamma_3$.
\end{claim}

\bigskip
The case $\phi_0(\bar P_3)>0$ is way more elaborate, and this is where the main novelty of our
contribution lies. We begin with the following claim.
The intuition behind it is as follows.
Identity~(\ref{3.9}) gives a linear relation between triangle and edge
densities via a vertex. By Claim~\ref{inductive} we know that (\ref{3.9})
also holds for the subgraph induced by the neighborhood
of almost every vertex $x\in V(G)$. If we average
this for all choices of $x$, then we get some linear relation between
the densities of $K_4$, $K_3$, and $K_2$ that has to hold for all
extremal homomorphisms. Repeating we get a linear relation for $K_5$, $K_4$, and
$K_3$, and so on.

\begin{claim}\label{KrRec} For every $r\ge 3$, we have
 \begin{equation}\label{eq:KrRec}
 \phi_0(K_r)=2(t-r+2)c\phi_0(K_{r-1})-(t-r+3)(t-r+2)c^2\phi_0(K_{r-2}).
 \end{equation}
 \end{claim}
 \begin{proof} We use induction on $r$. If $r=3$, then the identity relates
$b=\phi_0(K_3)$ and $a=\phi_0(\rho)$. Both of these parameters have been
explicitly expressed
in terms of $c$ and $t$ and the desired identity~(\ref{eq:KrRec}) can be
routinely checked.

Suppose that (\ref{eq:KrRec}) is true (for all extremal $\phi_0$).
Let us prove it for $r+1$.
Let $\phi^1\in S^1(\phi_0)$ be arbitrary
and let $\psi=\phi^1\pi^e$. By Claim~\ref{inductive}
we know that
$\psi(\rho)\in [1-\frac1{t-1},1-\frac1{t}]$. Let $\gamma=c(\psi(\rho))$,
where $c(x)$ is defined by~\req{c}, that is, $\gamma$ is the unique root
of
 \beq{gamma}
 2\left({t-1\choose 2}\gamma^2+(t-1)\gamma(1-(t-1)\gamma)\right)=\psi(\rho)
 \eeq
 with  $\gamma\ge 1/t$. We have that $\gamma=c/\phi^1(e)$. Indeed,
this value satisfies~\req{gamma} by~(\ref{3.13})
and is at least $1/t$ by~\req{3.11simple}.
 \hide{Indeed,
  $$
 2\left(\beta^2{t-1\choose 2}+\beta(t-1)(1-(t-1)\beta)\right)=\psi(\rho)
=\frac{z-\mu}{z^2}
=\frac{\frac d{2(t-1)c}-\frac{t}{4(t-1)}}{(\frac d{2(t-1)c})^2}.
 $$
 (and Mathematica confirms that this identity  indeed holds).
}
(An informal reason is that all derived inequalities are sharp for $\Phi$ and,
if we pass to a neighborhood of a vertex in some $H\in \C H_a$, then
its $t-2$ largest parts have the same (absolute) sizes as the
$t-1$ largest parts of $H$.)

By Claim~\ref{inductive}, we have that $t(\psi(\rho))=t-1$.
Thus, by the induction assumption,
$$
 \psi(K_{r})=2(t-r+1)\gamma\psi(K_{r-1}) - (t-r+2)(t-r+1) \gamma^2 \psi(K_{r-2}).
 $$
  If we now substitute $\gamma=c/\phi^1(e)$ and $\psi(K_s)=\phi^1(K_{s+1}^1)/
(\phi^1(e))^s$, cancel all occurrences of $(\phi^1(e))^{-r}$, and average the result,
we obtain exactly what we need.
  \end{proof}

Let us define $\hhh{r}(1)=1$ and, for $0\le x<1$,
 $$
 \hhh{r}(x)\df r!\left({t\choose r}c^r+{t\choose r-1}c^{r-1}(1-tc)\right),
 $$
 where $c=c(x)$ is again defined by~\req{c}. In other words, $\hhh{r}(x)$ is
the limiting density of $K_r$ in the graphs from $\C H_{x,n}$ as
$n\to\infty$. (In particular, $\hhh{3}$ is equal to our function $h$.) It is an upper bound on $g_r(x)$ and, as it was recently shown
by Reiher~\cite{reiher:Kr}, they are in fact equal.

Claim~\ref{KrRec} has the following useful corollary.

\begin{claim}\label{Kr} Let $r\ge 3$. Then $\phi_0(K_r)=\hhh{r}(a)$, that
is, each clique has the ``right" density. In particular,
$\phi_0(K_s)=0$ for $s\ge t+2$.\end{claim}

\begin{proof} This is true for $r=3$ as
$\phi_0(K_3)=g_3(a)$. The general case follows from
Claim~\ref{KrRec} by induction
on $r$.\end{proof}

Recall that we assume $\phi_0(\bar P_3)>0$ (as the case $\phi_0(\bar P_3)=0$ was
already tackled before). We need a few auxiliary results.
For a graph $F\in\C F_\ell^0$, let $F^{(1)}\in\C F_{\ell+1}^1$ be the $1$-flag
obtained by adding a new vertex $x$ that is connected to all vertices of $F$
(i.e., taking the join $F\vee K_1$) and labeling $x$ as $1$.

\begin{claim} \label{p3}
$\phi_0(\eval{\bar P_3^{(1)}}{1})>0$.
\end{claim}
\begin{proof}
By Claim~\ref{Kr} we have that $\phi_0(K_4)=\hhh{4}(a)$. When we substitute
this value into \req{3.6} we obtain a tight inequality except
for the extra term involving $\bar K_{1,3}$ (a triangle plus an isolated
vertex).  We conclude that
\begin{equation} \label{bark13}
\phi_0(\bar K_{1,3})=0.
\end{equation}
Inequality~\req{3.5}
is also used in the proof, so it has to be tight. Since we assumed that
$\phi_0(\bar P_3)>0$, we have that
$\phi_0(\eval{\bar
P_3^EK_3^E}{E})>0$, where $\bar P_3^E$ is the unique $E$-flag on $\bar P_3$.
But
 $$
 \eval{\bar P_3K_3^E}{E}= \frac 14\bar K_{1,3}+\frac
13\eval{\bar P_3^{(1)}}{1},
 $$
 and the claim follows.
\end{proof}

The two graphs in Figure \ref{graphs}, called $G_1$ and $G_2$, will play
a special role.
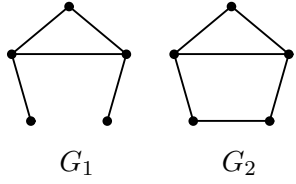
\begin{figure}[tb]
\begin{center}
\setlength{\unitlength}{0.254mm}
\begin{picture}(149,88)(13,-101)
        \special{color rgb 0 0 0}\allinethickness{0.254mm}\special{sh 0.99}\put(25,-75){\ellipse{4}{4}} 
        \special{color rgb 0 0 0}\allinethickness{0.254mm}\special{sh 0.99}\put(65,-75){\ellipse{4}{4}} 
        \special{color rgb 0 0 0}\allinethickness{0.254mm}\special{sh 0.99}\put(75,-40){\ellipse{4}{4}} 
        \special{color rgb 0 0 0}\allinethickness{0.254mm}\special{sh 0.99}\put(15,-40){\ellipse{4}{4}} 
        \special{color rgb 0 0 0}\allinethickness{0.254mm}\special{sh 0.99}\put(45,-15){\ellipse{4}{4}} 
        \special{color rgb 0 0 0}\allinethickness{0.254mm}\path(15,-40)(45,-15) 
        \special{color rgb 0 0 0}\allinethickness{0.254mm}\path(45,-15)(75,-40) 
        \special{color rgb 0 0 0}\allinethickness{0.254mm}\special{sh 0.99}\put(110,-75){\ellipse{4}{4}} 
        \special{color rgb 0 0 0}\allinethickness{0.254mm}\special{sh 0.99}\put(150,-75){\ellipse{4}{4}} 
        \special{color rgb 0 0 0}\allinethickness{0.254mm}\special{sh 0.99}\put(160,-40){\ellipse{4}{4}} 
        \special{color rgb 0 0 0}\allinethickness{0.254mm}\special{sh 0.99}\put(100,-40){\ellipse{4}{4}} 
        \special{color rgb 0 0 0}\allinethickness{0.254mm}\special{sh 0.99}\put(130,-15){\ellipse{4}{4}} 
        \special{color rgb 0 0 0}\allinethickness{0.254mm}\path(100,-40)(130,-15) 
        \special{color rgb 0 0 0}\allinethickness{0.254mm}\path(130,-15)(160,-40) 
        \special{color rgb 0 0 0}\allinethickness{0.254mm}\path(110,-75)(150,-75) 
        \special{color rgb 0 0 0}\allinethickness{0.254mm}\path(15,-40)(25,-75) 
        \special{color rgb 0 0 0}\allinethickness{0.254mm}\path(65,-75)(75,-40) 
        \special{color rgb 0 0 0}\allinethickness{0.254mm}\path(100,-40)(110,-75) 
        \special{color rgb 0 0 0}\allinethickness{0.254mm}\path(160,-40)(150,-75) 
        \special{color rgb 0 0 0}\allinethickness{0.254mm}\path(15,-40)(75,-40) 
        \special{color rgb 0 0 0}\allinethickness{0.254mm}\path(100,-40)(160,-40) 
        \special{color rgb 0 0 0}\put(40,-101){\shortstack{$G_1$}} 
        \special{color rgb 0 0 0}\put(125,-101){\shortstack{$G_2$}} 
        \special{color rgb 0 0 0} 
\end{picture}
\caption{\ \ Exceptional graphs\ \ \label{graphs}}
\end{center}
\end{figure}

\begin{claim}\label{G1G2} $\phi_0(G_1)=\phi_0(G_2)=0$.
\end{claim}

\begin{proof}
We apply the same strategy (although with much more involved
calculations) as the one used to prove \req{bark13}. Namely,
we make up an analog of \req{3.6} that is
tight on extremal homomorphisms and such
that the ``overall slackness'' involved will cover $G_1$ and $G_2$.

Form the element $f^E\in\scr F^E_4$ as follows:
$$
f^E\df \frac 12P_4^{E,c}-\frac 12P_4^{E,b} -F^E,
$$
where $P_4^{E,c}, P_4^{E,b}, F^E\in\C F_4^E$ are shown on Figure \ref{f4}.
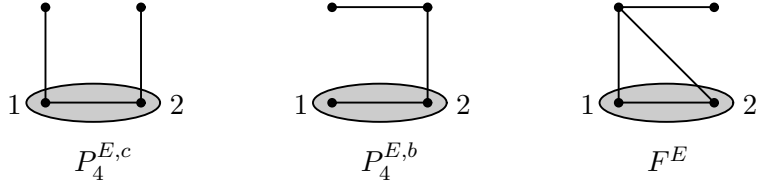
\begin{figure}[tb]
\begin{center}
\setlength{\unitlength}{0.254mm}
\begin{picture}(403,88)(5,-101)
        \special{color rgb 0 0 0}\allinethickness{0.254mm}\special{sh 0.2}\put(200,-65){\ellipse{70}{20}} 
        \special{color rgb 0 0 0}\allinethickness{0.254mm}\special{sh 0.99}\put(175,-15){\ellipse{4}{4}} 
        \special{color rgb 0 0 0}\allinethickness{0.254mm}\special{sh 0.99}\put(175,-65){\ellipse{4}{4}} 
        \special{color rgb 0 0 0}\allinethickness{0.254mm}\special{sh 0.99}\put(225,-15){\ellipse{4}{4}} 
        \special{color rgb 0 0 0}\allinethickness{0.254mm}\special{sh 0.99}\put(225,-65){\ellipse{4}{4}} 
        \special{color rgb 0 0 0}\put(155,-71){\shortstack{$1$}} 
        \special{color rgb 0 0 0}\put(240,-71){\shortstack{$2$}} 
        \special{color rgb 0 0 0}\put(190,-101){\shortstack{$P_4^{E,b}$}} 
        \special{color rgb 0 0 0}\allinethickness{0.254mm}\path(175,-15)(225,-15) 
        \special{color rgb 0 0 0}\allinethickness{0.254mm}\path(175,-65)(225,-65) 
        \special{color rgb 0 0 0}\allinethickness{0.254mm}\path(225,-15)(225,-65) 
        \special{color rgb 0 0 0}\allinethickness{0.254mm}\special{sh 0.2}\put(350,-65){\ellipse{70}{20}} 
        \special{color rgb 0 0 0}\allinethickness{0.254mm}\special{sh 0.99}\put(325,-15){\ellipse{4}{4}} 
        \special{color rgb 0 0 0}\allinethickness{0.254mm}\special{sh 0.99}\put(325,-65){\ellipse{4}{4}} 
        \special{color rgb 0 0 0}\allinethickness{0.254mm}\special{sh 0.99}\put(375,-15){\ellipse{4}{4}} 
        \special{color rgb 0 0 0}\allinethickness{0.254mm}\special{sh 0.99}\put(375,-65){\ellipse{4}{4}} 
        \special{color rgb 0 0 0}\put(305,-71){\shortstack{$1$}} 
        \special{color rgb 0 0 0}\put(390,-71){\shortstack{$2$}} 
        \special{color rgb 0 0 0}\put(340,-101){\shortstack{$F^E$}} 
        \special{color rgb 0 0 0}\allinethickness{0.254mm}\path(325,-15)(325,-65) 
        \special{color rgb 0 0 0}\allinethickness{0.254mm}\path(325,-15)(375,-65) 
        \special{color rgb 0 0 0}\allinethickness{0.254mm}\path(325,-65)(375,-65) 
        \special{color rgb 0 0 0}\allinethickness{0.254mm}\special{sh 0.2}\put(50,-65){\ellipse{70}{20}} 
        \special{color rgb 0 0 0}\allinethickness{0.254mm}\special{sh 0.99}\put(25,-15){\ellipse{4}{4}} 
        \special{color rgb 0 0 0}\allinethickness{0.254mm}\special{sh 0.99}\put(25,-65){\ellipse{4}{4}} 
        \special{color rgb 0 0 0}\allinethickness{0.254mm}\special{sh 0.99}\put(75,-15){\ellipse{4}{4}} 
        \special{color rgb 0 0 0}\allinethickness{0.254mm}\special{sh 0.99}\put(75,-65){\ellipse{4}{4}} 
        \special{color rgb 0 0 0}\put(5,-71){\shortstack{$1$}} 
        \special{color rgb 0 0 0}\put(90,-71){\shortstack{$2$}} 
        \special{color rgb 0 0 0}\put(40,-101){\shortstack{$P_4^{E,c}$}} 
        \special{color rgb 0 0 0}\allinethickness{0.254mm}\path(25,-15)(25,-65) 
        \special{color rgb 0 0 0}\allinethickness{0.254mm}\path(25,-65)(75,-65) 
        \special{color rgb 0 0 0}\allinethickness{0.254mm}\path(75,-15)(75,-65) 
        \special{color rgb 0 0 0}\allinethickness{0.254mm}\path(325,-15)(375,-15) 
        \special{color rgb 0 0 0} 
\end{picture}
\caption{\ \ Some $E$-flags\ \ \label{f4}}
\end{center}
\end{figure}
Since \req{3.5} is tight,
$$
\rn{\phi_0^E}(K_3^E) <\frac 13h_t'(a)\quad \Longrightarrow\quad \rn{\phi_0^E}(\bar P_3^E)=0\ \text{a.e.}
$$
and, since both $P_4^{E,b}$ and $F^E$ contain $\bar P_3^E$, this implies that
 \beq{fEK3E}
\rn{\phi_0^E}(K_3^E) <\frac 13h_t'(a)\quad \Longrightarrow\quad
\rn{\phi_0^E}(f^E)\geq 0\ \text{a.e.}
 \eeq
(Recall that $h_t$ is just the restriction of $\hh$ to the interval $[1-\frac
1t, 1-\frac 1{t+1}]$
as defined by~\req{ht}.) Thus, by~\req{3.3}, we can multiply the left-hand side of~\req{fEK3E} by $f^E$,
obtaining a true inequality. If we apply the averaging operator $\eval{\dots}E$ to this new
inequality, we get that
\begin{equation} \label{e_relation}
\phi_0(\eval{f^EK_3^E}{E}) \leq \frac 13h_t'(a)\phi_0(\eval{f^E}{E}).
\end{equation}
Next, similarly to \cite[(3.4)]{razborov:08} but multiplying
\cite[(3.2)]{razborov:08} (i.e.\ our formula~\req{3.2} which is equality a.e.)
by $K_3^1$ rather than by $e$, we obtain
\begin{equation} \label{one_relation}
\phi_0(3\eval{(K_3^1)^2}{1}-2h_t'(a)\eval{eK_3^1}{1})= b(3b-2ah_t'(a)).
\end{equation}
Subtracting \req{one_relation} from \req{e_relation} multiplied by 3, and re-grouping terms, we obtain
\begin{equation} \label{id4}
3\phi_0(\eval{f^EK_3^E}{E}-\eval{(K_3^1)^2}{1}) + h_t'(a)\phi_0(2\eval{eK_3^1}{1}-\eval{f^E}{E})\leq b(2ah_t'(a) - 3b).
\end{equation}

But we also have
\begin{equation} \label{id5}
2\eval{eK_3^1}{1}-\eval{f^E}{E} =\frac 43K_3+\frac 23K_4-\frac 13\bar K_{1,3}
\end{equation}
and
\begin{equation} \label{id6}
\eval{f^EK_3^E}{E}-\eval{(K_3^1)^2}{1} \geq \frac 1{60}(G_1+G_2) -\left(\frac 12K_4+\frac 13\rho K_3+\frac 16K_5\right).
\end{equation}
Substituting these relations into \req{id4}, and using Claim \ref{Kr},
we conclude by~\req{bark13} that
\begin{eqnarray*}
\frac 1{20}\phi_0(G_1+G_2) &\leq& b(2ah_t'(a)-3b)
 -h_t'(a)\left(\frac 43b+\frac 23\,\hhh{4}(a)\right)\\
 &+&\left(\frac 32\,\hhh{4}(a)+ab+\frac 12\, \hhh{5}(a)\right)\ =\ 0.
\end{eqnarray*}
Claim \ref{G1G2} is proved.
\end{proof}

\begin{lemma}\label{lm:5comb} Let $G$ be a graph on
$V=\{x_1,x_2,x_3,y,z\}$ with the following properties.
The vertices $x_1,x_2,x_3$ induce $\bar P_3$ with
$x_1x_2\in E(G)$, $y$ is adjacent to each $x_i$ and $z$
is non-adjacent to at least one $x_i$.

If $yz\not\in E(G)$, then $G$ contains
$\bar K_{1,3}$ as  an induced subgraph or $G$ is isomorphic to
$G_1$ or $G_2$.\end{lemma}

\begin{proof} If $zx_1,zx_2\in E(G)$, then $zx_3\not\in E(G)$ and
$G-y\cong \bar K_{1,3}$. If $zx_1,zx_2\not\in E(G)$, then $G-x_3\cong
\bar K_{1,3}$. So we can assume without loss of generality that $zx_1\in
E(G)$ and
$zx_2\not\in E(G)$. Now, if $zx_3\not\in  E(G)$, then $G$ is isomorphic to
$G_1$;
otherwise $G\cong G_2$.
\end{proof}

Now we are ready to put everything together. The next argument would look
particularly simple and elegant in genuinely flag-algebraic notation, but
it would require introducing some more notions and techniques, notably
upward operators (\cite[Section 2.3.1]{razborov:07}) and relating
extensions for different types (\cite[Theorem 3.17]{razborov:07}). We prefer
not to indulge into this endeavor in the concluding part of our paper, so we
replace this with (admittedly, crude) translation to the finite world.

Let $\sigma$ be the 3-vertex type whose graph is $\bar P_3$ with labels $1$
and $2$ being adjacent.
Let $\{G_n\}$ converge to $\phi_0$ with $|V(G_n)|=n$.
By Claim~\ref{p3}, $G_n$ has
$\Omega(n^4)$
copies of $F_0\in\C F_4^0$, which denotes a triangle
with a pendant edge. Let $F_1\in \C F_4^1$ be obtained from $F_0$
by putting label 1 on a vertex of degree $2$. Let $F_3\in \C
F_4^\sigma$ be the (unique)
$\sigma$-flag that can be obtained from $F_1$ by adding labels 2 and 3.

Fix small positive constants $\e\gg \delta$. Let $X=\{x_1\in V(G_n)\mid
p(F_1,(G_n;x_1))>\e\}$. By counting
copies of $F_0$ in $G_n$, we conclude that
 $$
 2(\phi(F_0)+o(1)){n\choose 4}\le |X|  {n-1\choose 3}+(n-|X|)\e {n-1\choose 3},
 $$
 implying that $|X|\ge 2\e n$. An easy
counting shows that for every $x_1\in X$ there are
at least $\delta n^2$ pairs $(x_2,x_3)$ of vertices with
$p(F_3,(G_n;x_1,x_2,x_3))\ge \delta$.
Likewise, by~\req{3.11simple}, the set $Y=\{x_1\in
V(G_n)\mid p(e,(G_n;x_1))<
1-\e\}$ has size at least $(1-\e)n$. Thus $|X\cap Y|\ge \e n$
and there are at least $\e n\cdot  \delta n^2$ choices of $(x_1,x_2,x_3)$
such that $x_1\in X\cap Y$ and
$p(F_3,(G_n;x_1,x_2,x_3))\ge \delta$. Given such a triple, let
$V_1$ consist of all vertices of $G_n$ adjacent to all of
$x_1,x_2,x_3$ and let $V_2=V(G_n)\setminus V_1$.
We have $|V_1|\ge \delta(n-3)$.
Since $x_1\in Y$, we have
$|V_2|\ge \e(n-1)$
(note that all non-neighbors
of $x_1$ are in $V_2$). For each non-adjacent
$y\in V_1$ and $z\in V_2$, the $5$-set $\{x_1,x_2,x_3,y,z\}$
contains $G_1$, $G_2$ or $\bar K_{1,3}$ by Lemma~\ref{lm:5comb}.
By~\req{bark13} and Claim~\ref{G1G2}, each of these graphs has density
$o(1)$ in $G_n$. Thus there is a triple $(x_1,x_2,x_3)$ with $e(\bar
G[V_1,V_2])=o(n^2)$.

Fix one such choice. By taking a subsequence, we can assume
that $|V_i|/n$ tends to a limit $\alpha_i$ and that
$G_n[V_i]$ converges to some homomorphism $\phi_i$, for $i=1,2$.
Now, $\phi_0=\vee(\phi_1,\phi_2,\alpha_1,\alpha_2)$, where $\alpha_1\ge \delta$
and $\alpha_2\ge \e$ are bounded away from 0.

Let $i=1$ or $2$.
Each $\phi_i$ is an extremal homomorphism: for example, if there is
$\phi_1'$ with $\phi_1'(\rho)=\phi_1(\rho)$ and $\phi_1'(K_3)<
\phi_1(K_3)$, then $\vee(\phi_1',\phi_2,\alpha_1,\alpha_2)$ contradicts
the extremality of $\phi_0$. Since
$\phi_0(K_{t+2})=0$ and $\alpha_{3-i}>0$, we have $\phi_i(K_{t+1})=0$
for $i=1,2$.
Tur\'an's theorem implies
that $\phi_i(\rho)\le 1-\frac1t$. Thus we can apply
the (global) induction and conclude that  $\phi_i\in \Phi$.

We have proved so far that $\phi_0$ is a join of two elements from $\Phi$; in
particular, it has the form
\begin{equation} \label{phi0}
\phi_0=\vee(\underbrace{0,\dots,0}_{k \mbox{ \scriptsize times}},\psi_1,\psi_2; c_1,\ldots, c_k,
d_1, d_2),\quad \mbox{with }c_1,\ldots,c_k>0,
\end{equation}
where $\psi_1(K_3)=\psi_2(K_3)=0$. Let $\psi_i' \df \vee(0,0; p_i,1-p_i)$, where $p_i\leq 1/2$
satisfies $2p_i(1-p_i)=\psi_i(\rho)$. Since $\psi_i'(\rho) = \psi_i(\rho)$ and $\psi'(K_3)=\psi(K_3)\
(=0)$, after plugging $\psi_i'$ for $\psi_i$ into $\phi_0$, we will get another {\em extremal}
homomorphism
\begin{equation} \label{phiprime}
\phi_0'\df \vee(\underbrace{0,\dots,0}_{k+4 \mbox{ \scriptsize times}}; c_1,\ldots, c_k,
d_1p_1,d_1(1-p_1), d_2p_2, d_2(1-p_2)).
\end{equation}
The equality $\phi_0'(\bar P_3)=0$, as we already proved before, implies $\phi_0'\in\Phi$, that is, all non-zero weights in
\req{phi0} are equal except for possibly one that is allowed to be smaller than others. But
$\phi_0(\bar P_3)>0$ which implies that for at least one $\psi_i$, say, $\psi_1$, we have $d_1>0$ and
$0<p_1<1/2$. This already creates the exceptional weight $d_1p_1$ in \req{phiprime}; all others weights must
lie in $\{0, d_1(1-p_1)\}$. In particular, either $d_2=0$ or $p_2\in \{0,1/2\}$; in the first case $\psi_2$
can be crossed out from \req{phi0}, and in the second case $\psi_2=\psi_2'$ and it can be merged with the
first $k$ terms. Thus, $\phi_0\in\Phi$.

This finishes the proof of Theorem~\ref{th:Phi}.

\newcommand{\etalchar}[1]{$^{#1}$}
\providecommand{\bysame}{\leavevmode\hbox to3em{\hrulefill}\thinspace}
\providecommand{\MR}{\relax\ifhmode\unskip\space\fi MR }
\providecommand{\MRhref}[2]{%
  \href{http://www.ams.org/mathscinet-getitem?mr=#1}{#2}
}
\providecommand{\href}[2]{#2}


\begin{thebibliography}{DHM{\etalchar{+}}13}

\bibitem[AR13]{aristoff+radin:13}
D.~Aristoff and C.~Radin, \emph{Emergent structures in large networks}, J.
  Appl. Probab. \textbf{50} (2013), 883--888.

\bibitem[BCL{\etalchar{+}}08]{BCLSV:08}
C.~Borgs, J.~Chayes, L.~Lov{\'a}sz, V.~T. S{\'o}s, and K.~Vesztergombi,
  \emph{Convergent sequences of dense graphs {I}: {Subgraph} frequencies,
  metric properties and testing}, Adv.\ Math. \textbf{219} (2008), 1801--1851.

\bibitem[{Bol}76]{bollobas:76}
B.~{Bollob\'as}, \emph{On complete subgraphs of different orders}, Math.\
  Proc.\ Camb.\ Phil.\ Soc. \textbf{79} (1976), 19--24.

\bibitem[CD10]{chatterjee+dey:10}
S.~Chatterjee and P.~S. Dey, \emph{Applications of {S}tein's method for
  concentration inequalities}, Ann. Probab. \textbf{38} (2010), 2443--2485.

\bibitem[CD13]{chatterjee+diaconis:13}
S.~Chatterjee and P.~Diaconis, \emph{Estimating and understanding exponential
  random graph models}, Ann. Statist. \textbf{41} (2013), 2428--2461.

\bibitem[CD14]{chatterjee+dembo:14:arxiv}
S.~Chatterjee and A.~Dembo, \emph{Nonlinear large deviations}, E-print
  \texttt{arxiv:1401.3495}, 2014.

\bibitem[CKP{\etalchar{+}}13]{CKPSTY}
J.~Cummings, D.~Kr{\'a}l', F.~Pfender, K.~Sperfeld, A.~Treglown, and M.~Young,
  \emph{Monochromatic triangles in three-coloured graphs}, J.\ Combin.\ Theory\
  {\rm (B)} \textbf{103} (2013), 489--503.

\bibitem[CV11]{chatterjee+varadhan:11}
S.~Chatterjee and S.~R.~S. Varadhan, \emph{The large deviation principle for
  the {Erd\H{o}s-R\'enyi} random graph}, Europ.\ J.\ Combin. \textbf{32}
  (2011), 1000--1017.

\bibitem[DHM{\etalchar{+}}13]{das+huang+ma+naves+sudakov:13}
S.~Das, H.~Huang, J.~Ma, H.~Naves, and B.~Sudakov, \emph{A problem of {Erd\H
  os} on the minimum number of {$k$}-cliques}, J.\ Combin.\ Theory\ {\rm (B)}
  \textbf{103} (2013), 344--373.

\bibitem[EFR86]{erdos+frankl+rodl:86}
P.~Erd{\H{o}}s, P.~Frankl, and V.~R{\"o}dl, \emph{The asymptotic number of
  graphs not containing a fixed subgraph and a problem for hypergraphs having
  no exponent}, Graphs Combin. \textbf{2} (1986), 113--121.

\bibitem[Erd55]{erdos:55}
P.~Erd{\H{o}}s, \emph{Some theorems on graphs}, Riveon Lematematika \textbf{9}
  (1955), 13--17.

\bibitem[Erd62]{erdos:62}
\bysame, \emph{On a theorem of {R}ademacher-{T}ur\'an}, Illinois J. Math.
  \textbf{6} (1962), 122--127.

\bibitem[Erd67]{erdos:67a}
\bysame, \emph{Some recent results on extremal problems in graph theory.
  {R}esults}, Theory of Graphs (Internat. Sympos., Rome, 1966), Gordon and
  Breach, New York, 1967, pp.~117--123 (English); pp. 124--130 (French).

\bibitem[Erd69]{erdos:69:cpm}
\bysame, \emph{On the number of complete subgraphs and circuits contained in
  graphs.}, \v Casopis P\v est. Mat. \textbf{94} (1969), 290--296.

\bibitem[ES83]{erdos+simonovits:83}
P.~{Erd\H os} and M.~Simonovits, \emph{Supersaturated graphs and hypergraphs},
  Combinatorica \textbf{3} (1983), 181--192.

\bibitem[Fis89]{fisher:89}
D.~C. Fisher, \emph{Lower bounds on the number of triangles in a graph}, J.\
  Graph Theory \textbf{13} (1989), 505--512.

\bibitem[FK99]{frieze+kannan:99}
A.~Frieze and R.~Kannan, \emph{Quick approximation to matrices and
  applications}, Combinatorica \textbf{19} (1999), 175--220.

\bibitem[HHK{\etalchar{+}}13]{hatami+hladky+kral+norine+razborov:13}
H.~Hatami, J.~Hladk{\'y}, D.~Kr{\'a}l', S.~Norine, and A.~Razborov, \emph{On
  the number of pentagons in triangle-free graphs}, J.\ Combin.\ Theory\ {\rm
  (A)} \textbf{120} (2013), 722--732.

\bibitem[KS96]{komlos+simonovits:96}
J.~Koml\'os and M.~Simonovits, \emph{{Szemer\'edi's} regularity lemma and its
  applications to graph theory}, Combinatorics, {Paul} {Erd\H os} is Eighty
  (D.~Mikl\'os, V.~T. S\'os, and T.~Sz\H onyi, eds.), vol.~2, Bolyai Math.\
  Soc., 1996, pp.~295--352.

\bibitem[{Lov}92]{lovasz:cpe}
L.~{Lov\'asz}, \emph{Combinatorial problems and exercises}, North-Holland,
  1992.

\bibitem[{Lov}12]{lovasz:lngl}
\bysame, \emph{Large networks and graph limits}, Colloquium Publications,
  Amer.\ Math.\ Soc, 2012.

\bibitem[LS76]{lovasz+simonovits:76}
L.~Lov{\'a}sz and M.~Simonovits, \emph{On the number of complete subgraphs of a
  graph}, Proceedings of the Fifth British Combinatorial Conference (Univ.
  Aberdeen, Aberdeen, 1975) (Winnipeg, Man.), Utilitas Math., 1976,
  pp.~431--441. Congressus Numerantium, No. XV.

\bibitem[LS83]{lovasz+simonovits:83}
\bysame, \emph{On the number of complete subgraphs of a graph. {II}}, Studies
  in pure mathematics, Birkh\"auser, Basel, 1983, pp.~459--495.

\bibitem[LS06]{lovasz+szegedy:06}
L.~Lov{\'a}sz and B.~Szegedy, \emph{Limits of dense graph sequences}, J.\
  Combin.\ Theory\ {\rm (B)} \textbf{96} (2006), 933--957.

\bibitem[LZ14a]{lubetzky+zhao:rsa}
E.~Lubetzky and Y.~Zhao, \emph{On replica symmetry of large deviations in
  random graphs}, Accepted by Random Struct.\ Algorithms, 2014.

\bibitem[LZ14b]{lubetzky+zhao:14:arxiv}
\bysame, \emph{On the variational problem for upper tails in sparse random
  graphs}, E-print \texttt{arxiv:1402.6011}, 2014.

\bibitem[Man07]{mantel:07}
W.~Mantel, \emph{Problem 28}, Winkundige Opgaven \textbf{10} (1907), 60--61.

\bibitem[MM62]{moon+moser:62}
J.~W. Moon and L.~Moser, \emph{On a problem of {Tur\'an}}, Publ.\ Math.\ Inst.\
  Hungar.\ Acad. Sci. \textbf{7} (1962), 283--287.

\bibitem[Nik11]{nikiforov:11}
V.~Nikiforov, \emph{The number of cliques in graphs of given order and size},
  Trans.\ Amer.\ Math.\ Soc. \textbf{363} (2011), 1599--1618.

\bibitem[NS63]{nordhaus+stewart:63}
E.~A. Nordhaus and B.~M. Stewart, \emph{Triangles in an ordinary graph}, Can.\
  J.\ Math. \textbf{15} (1963), 33--41.

\bibitem[Pik11]{pikhurko:11}
O.~Pikhurko, \emph{The minimum size of $3$-graphs without four vertices
  spanning no or exactly three edges}, Europ.\ J.\ Combin. \textbf{23} (2011),
  1142--1155.

\bibitem[PV13]{pikhurko+vaughan:13}
O.~Pikhurko and E.~R. Vaughan, \emph{Minimum number of {$k$}-cliques in graphs
  with bounded independence number}, Combin.\ Probab.\ Computing \textbf{22}
  (2013), 910--934.

\bibitem[Raz07]{razborov:07}
A.~Razborov, \emph{Flag algebras}, J.\ Symb.\ Logic \textbf{72} (2007),
  1239--1282.

\bibitem[Raz08]{razborov:08}
\bysame, \emph{On the minimal density of triangles in graphs}, Combin.\
  Probab.\ Computing \textbf{17} (2008), 603--618.

\bibitem[Rei12]{reiher:Kr}
C.~Reiher, \emph{The clique density theorem}, E-print \texttt{arxiv:1212.2454},
  2012.

\bibitem[RRS14]{radin+ren+sadun:13}
C.~Radin, K.~Ren, and L.~Sadun, \emph{The asymptotics of large constrained
  graphs}, J. Phys. A \textbf{47} (2014), no.~17, 175001, 20.

\bibitem[RS78]{ruzsa+szemeredi:78}
I.~Z. Ruzsa and E.~{Szemer\'edi}, \emph{Triple systems with no six points
  carrying three triangles}, Combinatorics {II} (A.~Hajnal and V.~{S\'os},
  eds.), North Holland, Amsterdam, 1978, pp.~939--945.

\bibitem[RS13]{radin+sadun:13}
C.~Radin and L.~Sadun, \emph{Phase transitions in a complex network}, J. Phys.
  A \textbf{46} (2013), no.~30, 305002, 12.

\bibitem[RY11]{radin+yin:11:arxiv}
C.~Radin and M.~Yin, \emph{Phase transitions in exponential random graphs},
  E-Print \texttt{arxiv.org:1108.0649}, 2011.

\bibitem[Sim68]{simonovits:68}
M.~Simonovits, \emph{A method for solving extremal problems in graph theory,
  stability problems}, Theory of Graphs (Proc. Colloq., Tihany, 1966), Academic
  Press, 1968, pp.~279--319.

\bibitem[{Tur}41]{turan:41}
P.~{Tur\'an}, \emph{On an extremal problem in graph theory (in {Hungarian})},
  Mat.\ Fiz.\ Lapok \textbf{48} (1941), 436--452.
\end{thebibliography}
\end{document}